\documentclass[english]{article}
\usepackage[T1]{fontenc}
\usepackage[latin9]{inputenc}
\usepackage{refstyle}
\usepackage{enumitem}
\usepackage{amsmath}
\usepackage{amsthm}
\usepackage{amssymb}
\usepackage{setspace}
\makeatletter

\AtBeginDocument{\providecommand\subsecref[1]{\ref{subsec:#1}}}
\AtBeginDocument{\providecommand\secref[1]{\ref{sec:#1}}}
\AtBeginDocument{\providecommand\propref[1]{\ref{prop:#1}}}
\AtBeginDocument{\providecommand\thmref[1]{\ref{thm:#1}}}
\AtBeginDocument{\providecommand\lemref[1]{\ref{lem:#1}}}
\AtBeginDocument{\providecommand\corref[1]{\ref{cor:#1}}}
\AtBeginDocument{\providecommand\exaref[1]{\ref{exa:#1}}}
\let\SF@@footnote\footnote
\def\footnote{\ifx\protect\@typeset@protect
    \expandafter\SF@@footnote
  \else
    \expandafter\SF@gobble@opt
  \fi
}
\expandafter\def\csname SF@gobble@opt \endcsname{\@ifnextchar[
  \SF@gobble@twobracket
  \@gobble
}
\edef\SF@gobble@opt{\noexpand\protect
  \expandafter\noexpand\csname SF@gobble@opt \endcsname}
\def\SF@gobble@twobracket[#1]#2{}
\RS@ifundefined{subsecref}
  {\newref{subsec}{name = \RSsectxt}}
  {}
\RS@ifundefined{thmref}
  {\def\RSthmtxt{theorem~}\newref{thm}{name = \RSthmtxt}}
  {}
\RS@ifundefined{lemref}
  {\def\RSlemtxt{lemma~}\newref{lem}{name = \RSlemtxt}}
  {}

\theoremstyle{plain}
\newtheorem{thm}{\protect\theoremname}[section]
\theoremstyle{definition}
\newtheorem{defn}[thm]{\protect\definitionname}
\theoremstyle{plain}
\newtheorem{lem}[thm]{\protect\lemmaname}
\theoremstyle{plain}
\newtheorem{cor}[thm]{\protect\corollaryname}
\newlist{casenv}{enumerate}{4}
\setlist[casenv]{leftmargin=*,align=left,widest={iiii}}
\setlist[casenv,1]{label={{\itshape\ \casename} \arabic*.},ref=\arabic*}
\setlist[casenv,2]{label={{\itshape\ \casename} \roman*.},ref=\roman*}
\setlist[casenv,3]{label={{\itshape\ \casename\ \alph*.}},ref=\alph*}
\setlist[casenv,4]{label={{\itshape\ \casename} \arabic*.},ref=\arabic*}
\theoremstyle{remark}
\newtheorem{rem}[thm]{\protect\remarkname}
\theoremstyle{definition}
\newtheorem{example}[thm]{\protect\examplename}
\theoremstyle{plain}
\newtheorem{prop}[thm]{\protect\propositionname}

\@ifundefined{date}{}{\date{}}

\AtBeginDocument{}
\AtBeginDocument{\providecommand\propref[1]{\ref{prop:#1}}}
\RS@ifundefined{propref}{\newref{prop}{name =Proposition~,names = Propositions~}}{}
\AtBeginDocument{\providecommand\corref[1]{\ref{cor:#1}}}
\RS@ifundefined{corref}{\newref{cor}{name =Corollary~,names = Corollaries~}}{}
\RS@ifundefined{exaref}{\newref{exa}{name =Example~,names = Examples~}}{}
\RS@ifundefined{eqqref}{\newref{eqq}{name =Equality~,names = Equality~}}{}
\AtBeginDocument{\renewcommand\subsecref[1]{Section~\ref{subsec:#1}}}
\AtBeginDocument{\renewcommand\secref[1]{Section~\ref{sec:#1}}}

\def\@fnsymbol#1{\ensuremath{\ifcase#1\or *\or **\or \ddagger\or
   \mathsection\or \mathparagraph\or \|\or \dagger\dagger
   \or \ddagger\ddagger \else\@ctrerr\fi}}

\usepackage{enumitem, hyperref}
\makeatletter
\def\namedlabel#1#2{\begingroup
    #2%
    \def\@currentlabel{#2}%
    \phantomsection\label{#1}\endgroup
}

\usepackage{tikz}

\usepackage{tocbibind}
\usepackage{youngtab}
\usepackage{authblk}
\usetikzlibrary{matrix}
\usepackage{bbding}
\usepackage{ytableau}
\usepackage{amsmath}
\usepackage{geometry}

\tikzset{   pt/.style={insert path={node[scale=2]{.}}},   dnup/.style={insert path={ [pt] .. controls +(0,1) and +(0,-1) .. +(#1,2) [pt]}},   dndn/.style={insert path={ [pt] .. controls +(0,0.25) and +(0,0.25) .. +(#1,0) [pt]}},   upup/.style={insert path={ [pt] .. controls +(0,-0.25) and +(0,-0.25) .. +(#1,0) [pt]}}, upup2/.style={insert path={ [pt] .. controls +(0,-0.5) and +(0,-0.5) .. +(#1,0) [pt]}}, }

\newcommand{\Rt}{\widetilde{\mathcal{R}}_E}
\newcommand{\Lt}{\widetilde{\mathcal{L}}_E}

\DeclareMathOperator{\Hom}{Hom}

\DeclareMathOperator{\im}{\mathsf{im}}

\DeclareMathOperator{\PT}{\mathcal{PT}}

\DeclareMathOperator{\C}{\mathcal{C}}

\DeclareMathOperator{\IS}{\mathcal{IS}}
\DeclareMathOperator{\T}{\mathcal{T}}

\DeclareMathOperator{\Rc}{\mathcal{R}}
\DeclareMathOperator{\Lc}{\mathcal{L}}
\DeclareMathOperator{\Hc}{\mathcal{H}}

\DeclareMathOperator{\Jc}{\mathcal{J}}
\DeclareMathOperator{\rank}{rank}

\DeclareMathOperator{\op}{op}

\DeclareMathOperator{\Op}{\mathcal{O}}
\DeclareMathOperator{\OF}{\mathcal{OF}}
\DeclareMathOperator{\POI}{\mathcal{POI}}
\DeclareMathOperator{\IO}{\mathcal{IO}}
\DeclareMathOperator{\IC}{\mathcal{IC}}

\def\RSlemtxt{Lemma~}
\def\RSthmtxt{Theorem~}

\usepackage{babel}
\providecommand{\corollaryname}{Corollary}
  \providecommand{\definitionname}{Definition}
  \providecommand{\examplename}{Example}
  \providecommand{\lemmaname}{Lemma}
  \providecommand{\propositionname}{Proposition}
  \providecommand{\remarkname}{Remark}
 \providecommand{\casename}{Case}
\providecommand{\theoremname}{Theorem}

\makeatother

\usepackage{babel}
\providecommand{\casename}{Case}
\providecommand{\corollaryname}{Corollary}
\providecommand{\definitionname}{Definition}
\providecommand{\examplename}{Example}
\providecommand{\lemmaname}{Lemma}
\providecommand{\propositionname}{Proposition}
\providecommand{\remarkname}{Remark}
\providecommand{\theoremname}{Theorem}

\begin{document}
\title{Algebras of reduced $E$-Fountain semigroups and the generalized ample
identity II}
\author{Itamar Stein\thanks{Mathematics Unit, Shamoon College of Engineering, 77245 Ashdod, Israel}\\
\Envelope \, Steinita@gmail.com}
\maketitle
\begin{abstract}
We study the generalized right ample identity, introduced by the author
in a previous paper. Let $S$ be a reduced $E$-Fountain semigroup
which satisfies the congruence condition. We can associate with $S$
a small category $\mathcal{C}(S)$ whose set of objects is identified
with the set $E$ of idempotents and its morphisms correspond to elements
of $S$. We prove that $S$ satisfies the generalized right ample
identity if and only if every element of $S$ induces a homomorphism
of left $S$-actions between certain classes of generalized Green's
relations. In this case, we interpret the associated category $\mathcal{C}(S)$
as a discrete form of a Peirce decomposition of the semigroup algebra.
We also give some natural examples of semigroups satisfying this identity.
\end{abstract}

\section{Introduction}

Let $S$ be a semigroup and let $\Bbbk$ be a field. It is of interest
to study the semigroup algebra $\Bbbk S$. In many cases we can associate
with $S$ a category $\mathcal{C}$ such that the semigroup algebra
$\Bbbk S$ is isomorphic to the category algebra $\Bbbk\mathcal{C}$.
In \cite{Stein2021} such a result is obtained for a certain class
of $E$-Fountain semigroups and in this context the generalized right
ample identity is introduced. Given a subset of idempotents $E$ of
$S$ we can define two equivalence relations $\Lt$ and $\Rt$ on
$S$. We say that $a\,\Lt\,b$ ($a\,\Rt\,b$) if $a$ and $b$ have
the same set of right (respectively, left) identities from $E$. The
semigroup $S$ is called reduced $E$-Fountain if every $\Lt$ and
$\Rt$-class contains a (unique) idempotent from $E$ and $ef=e$
if and only if $fe=e$ for every $e,f\in E$. In this case, for every
$a\in S$ we denote by $a^{\ast}$ ($a^{+})$ the unique idempotent
from $E$ in the $\Lt$-class (respectively, $\Rt$-class) of $a$.
If in addition, $\Lt$ and $\Rt$ are right and left congruences respectively
then we say that $S$ satisfies the congruence condition and we can
associate a small category $\mathcal{C}(S)$ with the semigroup $S$.
The set of objects of $\mathcal{C}(S)$ is identified with the set
$E$ and every element of $S$ corresponds to a morphism from $a^{\ast}$
to $a^{+}$(full details will be given in \subsecref{Semigroups}).
A reduced $E$-Fountain semigroup which satisfies the congruence condition
is also called a DRC-semigroup in the literature and this class is
under current active research (see \cite{Jones2021,Wang2022,wang2023munn}).
Let $S$ be a reduced $E$-Fountain semigroup which satisfies the
congruence condition. We say that $S$ satisfies the generalized right
ample identity if $\left(e\left(a\left(eaf\right)^{\ast}\right)^{+}\right)^{\ast}=\left(a\left(eaf\right)^{\ast}\right)^{+}$
holds for every $a\in S$ and $e,f\in E$. Under some mild conditions
on $S$, it is proved in \cite{Stein2021} that if the generalized
right ample identity holds then $\Bbbk S$ is isomorphic to the algebra
$\Bbbk\mathcal{C}(S)$ of its associated category. This is a generalization
of several results (\cite{Guo2012,Margolis2018A,Solomon1967,Stein2017,Stein2018erratum,Steinberg2006})
that were useful in the study of algebras of various classes of semigroups
such as inverse semigroups (\cite{Steinberg2008} and \cite[Part IV]{Steinberg2016})
and monoids of partial functions (\cite{Stein2016,Stein2019,Stein2020})
- see also \cite{Guo2018,Wang2017} for different but related approaches.
This motivates consideration of the generalized right ample identity.
However, at first sight this identity could seem unnatural and \cite{Stein2021}
suggests only one motivating example - the Catalan monoid. The goal
of this paper is to obtain a deeper understanding of the generalized
right ample identity. We will show that it has a more natural description
and we will explore additional examples of semigroups which satisfy
it. We remark that recently, Wang introduced \cite{wang2023munn}
a different generalization of the ample condition for the class of
reduced $E$-Fountain semigroups which satisfy the congruence condition
($=$DRC-semigroups). See \cite[Section 2]{wang2023munn} for a comparison
between the two different generalizations of the ample condition.

Consider a reduced $E$-Fountain semigroup which satisfies the congruence
condition. The semigroup $S$ acts partially on the left of every
$\Lt$-class. For every $a\in S$, right multiplication by $a$ induces
a function from the $\Lt$-class of $a^{+}$ to the $\Lt$-class $a^{\ast}$.
In \subsecref{Green'sRelations} we show that this function is a homomorphism
of partial left $S$-actions if and only if the generalized right
ample identity holds. In this case we can identify the morphisms of
the associated category $\mathcal{C}(S)$ with all such homomorphisms.
In \subsecref{Modules} we assume that $S$ is finite and turn to
consider $\Bbbk S$-modules. Under a certain condition on $S,$ we
show that if the generalized right ample condition holds then the
left modules formed by linear combinations of elements of an $\Lt$-class
are projective modules of $\Bbbk S$, thus extending a result from
\cite{Margolis2021}. In this case the linear category $\mathcal{L}$
generated from $\mathcal{C}(S)^{\op}$ is the linear category of $\Bbbk S$-module
homomorphisms between these projective modules. If we consider the
right Peirce decomposition of $\Bbbk S$ into a direct sum of the
above-mentioned projective modules (see \cite[Chapter 3, Section 7]{MRJacobson1956})
and intuitively identify it with the linear category $\mathcal{L}$
then we can think of the category $\mathcal{C}(S)$ as a ``discrete''
type of a Peirce decomposition for the semigroup $S$ itself. In \secref{Examples}
we give two additional examples of semigroups which satisfy the generalized
right ample identity, namely, bounded linear operators on a Hilbert
space and order preserving functions with a fixed point. We study
in greater detail the latter one. In particular, we prove that its
algebra is semisimple and isomorphic to the algebra of the (inverse)
semigroup of all order-preserving partial permutations.

\textbf{Acknowledgments:} The author thanks Professor Stuart Margolis
for several helpful conversations and in particular for suggesting
\propref{Iso_of_algebras_fixed_point}. The author is grateful to
the referee for his/her valuable comments and suggestions which improved
the paper significantly.

\section{Preliminaries }

\subsection{\label{subsec:Semigroups}Semigroups}

Let $S$ be a semigroup and let $S^{1}=S\cup\{1\}$ be the monoid
formed by adjoining a formal unit element. Recall that Green's preorders
$\leq_{\Rc}$, $\leq_{\Lc}$ and $\leq_{\Jc}$ are defined by:
\begin{align*}
a\leq_{\Rc}b\iff & aS^{1}\subseteq bS^{1}\\
a\leq_{\Lc}b\iff & S^{1}a\subseteq S^{1}b\\
a\leq_{\Jc}b\iff & S^{1}aS^{1}\subseteq S^{1}bS^{1}.
\end{align*}
The associated Green's equivalence relations on $S$ are denoted by
$\Rc$, $\Lc$ and $\Jc$. Recall also that $\Hc=\Rc\cap\Lc$. It
is well-known that $\Lc$ ($\Rc$) is a right congruence (respectively,
left congruence). We can define a partial order on the set of idempotents
$E(S)$ of $S$ according to
\[
e\leq f\iff ef=e=fe.
\]

We assume familiarity with additional basic notions of semigroup theory
that can be found in standard textbooks such as \cite{Howie1995}.

Let $E\subseteq E(S)$ be a subset of idempotents of $S$. We define
two equivalence relations $\Lt$ and $\Rt$ on $S$ by 
\[
a\,\Lt\,b\iff(\forall e\in E\quad ae=a\iff be=b)
\]
\[
a\,\Rt\,b\iff(\forall e\in E\quad ea=a\iff eb=b).
\]

These relations are one type of ``generalized'' Green's relations
(see \cite{Gould2010}). Note that $\Lc\subseteq\Lt$ and $\Rc\subseteq\Rt$.
\begin{defn}
The semigroup $S$ is called \emph{$E$-Fountain} if every $\Lt$-class
contains an idempotent from $E$ and every $\Rt$-class contains an
idempotent from $E$. 
\end{defn}

We remark that this property is also called ``$E$-semiabundant''
in the literature.
\begin{defn}
An $E$-Fountain semigroup $S$ is called \emph{reduced} if 
\[
\mbox{\ensuremath{ef=e\iff fe=e}}
\]
 for every $e,f\in E$. Equivalently, $S$ is reduced if $\leq_{\Lc}=\leq_{\Rc}$
when restricted to the set $E$.
\end{defn}

A reduced \emph{$E$-}Fountain\emph{ }semigroup is called a ``DR
semigroup'' in \cite{Stokes2015}. In such a semigroup, every $\Lt$-class
contains a unique idempotent from $E$. The unique idempotent from
$E$ in the $\Lt$-class of $a\in S$ is denoted $a^{\ast}$. The
idempotent $a^{\ast}$ is the minimal right identity of $a$ from
the set $E$ (with respect to the natural partial order on idempotents).
In other words, if $e\in E$ satisfies $ae=a$ then $ea^{\ast}=a^{\ast}e=a^{\ast}$.
Dually, every $\Rt$-class contains a unique idempotent from $E$.
The unique idempotent from $E$ in the $\Rt$-class of $a$ which
is also its minimal left identity from the set $E$ is denoted $a^{+}$.
If $e\in E$ satisfies $ea=a$ then $ea^{+}=a^{+}e=a^{+}$. See \cite{Stokes2015}
for proofs and additional details. 
\begin{defn}
Let $S$ be a reduced $E$-Fountain semigroup. We say that $S$ satisfies
the \emph{congruence condition} if $\Lt$ is a right congruence and
$\Rt$ is a left congruence. 
\end{defn}

It is well-known that a reduced $E$-Fountain semigroup $S$ satisfies
the congruence condition if and only if the identities $(ab)^{\ast}=(a^{\ast}b)^{\ast}$
and $(ab)^{+}=(ab^{+})^{+}$ hold - see \cite[Lemma 4.1]{Gould2010}.
In this case we can define a category $\mathcal{C}(S)$ in the following
way. The objects are in one-to-one correspondence with elements of
the set $E$ of idempotents. The morphisms are in one-to-one correspondence
with elements of $S$. For every $a\in S$ the associated morphism
$C(a)$ has domain $a^{\ast}$ and range $a^{+}$. If the range of
$C(a)$ is the domain of $C(b)$ (that is, if $b^{\ast}=a^{+}$) the
composition $C(b)\cdot C(a)$ is defined to be $C(ba)$. The congruence
condition and the assumption $b^{\ast}=a^{+}$ implies that $(ba)^{+}=(ba^{+})^{+}=(bb^{\ast})^{+}=b^{+}$
and dually we can prove $(ba)^{\ast}=a^{\ast}$. It follows that $\mathcal{C}(S)$
is indeed a category - see \cite{Lawson1991} for additional details.

Let $S$ be a reduced $E$-Fountain semigroup which satisfies the
congruence condition. If $E$ is a subband, it is proved in \cite[Lemma 3.9]{Stein2021}
that $E$ is commutative. In this case, $S$ is called an \emph{$E$-Ehresmann}
semigroup, see \cite{Gould2010,Lawson1991} for more facts and examples. 

Recall that a semigroup $S$ is called \emph{inverse }if for every
$a\in S$ there exists a unique element denoted $a^{-1}$ such that
$aa^{-1}a=a$ and $a^{-1}aa^{-1}=a^{-1}$ (this is in fact a very
special case of an $E$-Ehresmann semigroup where $E=E(S)$). One
of the most important examples of an inverse semigroup is the symmetric
inverse monoid $\IS_{X}$ which consists of all \emph{partial }injective
transformations (also called \emph{partial permutations}) on the set
$X$.

Denote by $\PT_{X}$ the monoid of partial functions on a set $X$
- composing functions from right to left. Then $\PT_{X}^{\op}$ is
the monoid with the same underlying set but composing functions from
left to right. A \emph{partial left action} of $S$ on $X$ is a semigroup
homomorphism $\psi:S\to\PT_{X}$. Equivalently, we say that \emph{$S$
acts partially on the left of $X$} or that $X$ is a \emph{partial
left $S$-action}. Another convention is to write $s\cdot x$ instead
of $\psi(s)(x)$ for $s\in S$ and $x\in X$. Note that $s\cdot x$
might be undefined. A set $I\subseteq S$ is called a \emph{left ideal}
of $S$ if $sx\in I$ for every $s\in S$ and $x\in I$. If $I_{1},I_{2}$
are two left ideals such that $I_{1}\subseteq I_{2}$ then we can
define a partial $S$-action on the difference $I_{2}\setminus I_{1}$
by 
\[
s\cdot x=\begin{cases}
sx & sx\notin I_{1}\\
\text{undefined} & sx\in I_{1}
\end{cases}
\]
for every $s\in S$ and $x\in I_{2}\setminus I_{1}$. Assume that
$S$ acts partially on the left of two sets $X,Y$. A function $f:X\to Y$
is a \emph{homomorphism of partial left $S$-actions} if for every
$x\in X$ and $s\in S$, $s\cdot x$ is defined if and only if $s\cdot f(x)$
is defined and in this case 
\[
f(s\cdot x)=s\cdot f(x).
\]
There is a dual notion of course. A \emph{partial right action} of
$S$ on a set $X$ is a homomorphism $\psi:S\to\PT_{X}^{\op}$. Equivalently,
we say that $X$ is a \emph{partial right $S$-action}. Homomorphisms
of partial right $S$-actions are defined in the obvious dual way.
In this paper almost all actions will be left actions so we will omit
the word ``left''. If not stated explicitly otherwise, a partial
action is always a partial \emph{left }action.

\subsection{Linear categories }

Let $\mathcal{C}$ be a small category. We denote by $\mathcal{C}^{0}$
and $\mathcal{C}^{1}$ the sets of objects and morphisms of $\mathcal{C}$
respectively. For $e,f\in\mathcal{C}^{0}$ we write $\mathcal{C}(e,f)$
for the hom-set of all morphisms whose domain is $e$ and range is
$f$. 

Let $\Bbbk$ be a field. A $\Bbbk$-\emph{linear category} is a category
$L$ enriched over the category of $\Bbbk$-vector spaces. This means
that every hom-set of $L$ is a $\Bbbk$-vector space and the composition
of morphisms is a bilinear map with respect to the vector space operations.
Note that a $\Bbbk$-algebra is a $\Bbbk$-linear category with one
object. A \emph{functor} of $\mathbb{\Bbbk}$-linear categories is
a category functor which is also a linear transformation when restricted
to any hom-set. Let $\mathcal{C}$ be a category. We can form a $\Bbbk$-linear
category $\Bbbk[\mathcal{C}]$ in the following way. The objects of
$\Bbbk[\mathcal{C}]$ and $\mathcal{C}$ are identical, and for every
two objects $e,f$ the hom-set $\Bbbk[\mathcal{C}](e,f)$ is the $\Bbbk$-vector
space with basis $\mathcal{C}(e,f)$. In other words, $\Bbbk[\mathcal{C}](e,f)$
contains all formal linear combinations of morphisms from $\mathcal{C}(e,f)$.
The composition of morphisms in $\Bbbk[\mathcal{C}${]} is defined
naturally in the only way that extends the composition of $\mathcal{C}$
and forms a bilinear map.

\subsection{Algebras and modules}

Let $\Bbbk$ be a field and let $A$ be a finite dimensional and unital
$\Bbbk$-algebra. Recall that an $A$-module $P$ is called \emph{projective}
if the functor $\Hom_{A}(P,-)$ is an exact functor. Equivalently,
$P$ is projective if it is a direct summand of $A^{k}$ (for some
$k\in\mathbb{N}$) as an $A$-module. It is well-known that $Ae$
is a projective $A$-module for every idempotent $e\in A$. Let $E=\{e_{1},\ldots,e_{n}\}$
be a set of idempotents from $A$. Recall that $E$ is a complete
set of orthogonal idempotents if ${\displaystyle \sum_{i=1}^{n}e_{i}=1_{A}}$
and $e_{i}e_{j}=0$ if $i\neq j$. It is well-known that in this case
$A\simeq{\displaystyle \bigoplus_{i=1}^{n}Ae_{i}}$ as left $A$-modules
and $A\simeq{\displaystyle \bigoplus_{i,j=1}^{n}e_{i}Ae_{j}}$ as
$\Bbbk$-vector spaces. The first decomposition is called a \emph{right}
\emph{Peirce decomposition }of $A$ and the second is a \emph{two
sided} \emph{Peirce decomposition} - both are relative to the complete
set of orthogonal idempotents. Peirce decompositions are fundamental
in representation theory and ring theory (see \cite[Chapter III, Section 7]{MRJacobson1956},
\cite[Section 21]{Lam1991} or \cite[Section I.4]{Assem2006}). For
proofs and other basic facts on algebras and modules see \cite[Chapter I]{Assem2006}.

Fixing a complete set of orthogonal idempotents $E=\{e_{1},\ldots,e_{n}\}$,
we can associate with $A$ a $\Bbbk$-linear category $\mathcal{L}(A)$
whose objects are the projective modules of the form $Ae_{i}$ and
its morphisms are the $A$-module homomorphisms between them. In other
words, the hom-set of morphisms with domain $Ae_{i}$ and range $Ae_{j}$
is $\Hom_{A}(Ae_{i},Ae_{j})$. For the sake of simplicity, we call
$\mathcal{L}(A)$ also a Peirce decomposition of $A$ as they essentially
contain the same information (note that $\dim_{\Bbbk}\Hom_{A}(Ae_{i},Ae_{j})=\dim_{\Bbbk}e_{i}Ae_{j}$).

We will be mainly interested in this paper in semigroup and category
algebras. The \emph{semigroup algebra} $\mathbb{\Bbbk}S$ of a semigroup
$S$ is defined in the following way. It is a $\mathbb{\Bbbk}$-vector
space with basis the elements of $S$, that is, it consists of all
formal linear combinations
\[
\{k_{1}s_{1}+\ldots+k_{n}s_{n}\mid k_{i}\in\mathbb{\Bbbk},\,s_{i}\in S\}.
\]
The multiplication in $\mathbb{\Bbbk}S$ is the linear extension of
the semigroup multiplication. Note that in general $\Bbbk S$ might
not possess a unit element. 

The \emph{category algebra} $\mathbb{\Bbbk}\mathcal{C}$ of a (small)
category $\mathcal{C}$ is defined in the following way. It is a $\mathbb{\Bbbk}$-vector
space with the morphisms of $\mathcal{C}$ as a basis, that is, it
consists of all formal linear combinations
\[
\{k_{1}m_{1}+\ldots+k_{n}m_{n}\mid k_{i}\in\mathbb{\Bbbk},\,m_{i}\in\mathcal{C}^{1}\}.
\]
The multiplication in $\mathbb{\Bbbk}\mathcal{C}$ is the linear extension
of the following:
\[
m^{\prime}\cdot m=\begin{cases}
m^{\prime}m & \text{if \ensuremath{m^{\prime}m} is defined}\\
0 & \text{otherwise}.
\end{cases}
\]

If $\mathcal{C}$ has a finite number of objects then the unit element
of $\Bbbk\mathcal{C}$ is${\displaystyle \sum_{e\in\mathcal{C}^{0}}1_{e}}$
where $1_{e}$ is the identity morphism of the object $e$ of $\mathcal{C}$.

\section{The generalized right ample identity}

In \cite{Stein2021} the generalized right ample identity is defined
by several equivalent definitions, one of them being the following.
\begin{defn}
Let $S$ be a reduced $E$-Fountain semigroup which satisfies the
congruence condition. We say that the \emph{generalized right ample
identity (}or\emph{ generalized right ample condition) }holds in $S$
if the identity 
\[
\left(e\left(a\left(eaf\right)^{\ast}\right)^{+}\right)^{\ast}=\left(a\left(eaf\right)^{\ast}\right)^{+}
\]
holds for every $a\in S$ and $e,f\in E$.

If $S$ is an $E$-Ehresmann semigroup i.e., where $E$ is a commutative
subsemigroup (in fact, it is enough to assume that $E$ is a subsemigroup
with the other conditions being as stated) it is proved in \cite{Stein2021}
that the generalized right ample identity reduces to the ``standard''
well-studied \emph{right ample identity}
\[
ea=a(ea)^{\ast}
\]
for every $a\in S$ and $e\in E$. The motivation for defining the
generalized right ample identity comes from \thmref{iso_theorem}.
To state it we need two more definitions. We say that a relation $R$
on $S$ is \emph{principally finite} if for every $a\in S$ the set
$\{c\in S\mid c\,R\,a\}$ is finite. Define a relation $\trianglelefteq_{l}$
on $S$ by the rule that $a\trianglelefteq_{l}b$ if and only if $a=be$
for some $e\in E$. If $S$ is a reduced $E$-Fountain semigroup then
the relation $\trianglelefteq_{l}$ is reflexive but in general it
is not anti-symmetric nor transitive (even if the congruence condition
holds). We need the stronger assumption that $E$ is a commutative
subsemigroup to ensure that $\trianglelefteq_{l}$ is a partial order
(see \cite[Proposition 3.13]{Lawson1991}).
\end{defn}

\begin{thm}[{\cite[Theorems 4.2 and 4.4]{Stein2021}}]
\label{thm:iso_theorem}Let $S$ be a reduced $E$-Fountain semigroup
which satisfies the congruence condition and let $\mathcal{C}(S)$
be its associated category. Let $\Bbbk$ be a field and assume that
$\trianglelefteq_{l}$ is principally finite, so that the linear transformation
$\varphi:\Bbbk S\to\Bbbk\mathcal{C}(S)$ defined on basis elements
by 
\[
\varphi(a)=\sum_{c\trianglelefteq_{l}a}C(c)
\]
is well-defined. The linear transformation $\varphi$ is a homomorphism
of $\Bbbk$-algebras if and only if $S$ satisfies the generalized
right ample identity. If in addition we assume that $\trianglelefteq_{l}$
is contained in a principally finite partial order, then $\varphi$
is an isomorphism of $\Bbbk$-algebras if and only if $S$ satisfies
the generalized right ample identity.
\end{thm}

As already mentioned, in this paper we seek a deeper understanding
of the generalized right ample identity. The first step is to observe
that we can somewhat simplify this identity - this was also observed
independently by Peter R. Jones \cite[Proposition 2.9]{wang2023munn}.
\begin{lem}
Let $S$ be a reduced $E$-Fountain semigroup. The semigroup $S$
satisfies the generalized right ample identity if and only if it satisfies
the identity
\[
\left(e\left(a\left(ea\right)^{\ast}\right)^{+}\right)^{\ast}=\left(a\left(ea\right)^{\ast}\right)^{+}
\]

for every $e\in E$ and $a\in S$.
\end{lem}

\begin{proof}
If $S$ satisfies the generalized right ample identity, we can substitute
$f=a^{\ast}$ and obtain the simplified identity. In the other direction,
let $e,f\in E$ and $a\in S$ and set $a^{\prime}=af$. The simplified
identity implies 
\[
\left(e\left(a^{\prime}\left(ea^{\prime}\right)^{\ast}\right)^{+}\right)^{\ast}=\left(a^{\prime}\left(ea^{\prime}\right)^{\ast}\right)^{+}
\]
hence 
\begin{equation}
\left(e\left(af\left(eaf\right)^{\ast}\right)^{+}\right)^{\ast}=\left(af\left(eaf\right)^{\ast}\right)^{+}.\label{eq:GRA_Identity_proof}
\end{equation}

Now, $eaff=eaf$ so $(eaf)^{\ast}f=(eaf)^{\ast}$ by the definition
of $\Lt$. Therefore, $f(eaf)^{\ast}=(eaf)^{\ast}$ by the assumption
that $S$ is reduced so (\ref{eq:GRA_Identity_proof}) implies that
\[
\left(e\left(a\left(eaf\right)^{\ast}\right)^{+}\right)^{\ast}=\left(a\left(eaf\right)^{\ast}\right)^{+}
\]
as required.
\end{proof}

\subsection{\label{subsec:Green'sRelations}Green's relations and partial actions}

In this subsection we give a more informative interpretation for the
generalized right ample identity. It is natural to ask how much the
generalized Green's relations are similar to the standard ones. We
show that a certain property of Green's relations holds also for the
generalized relations if and only if the generalized right ample identity
holds.

Let $S$ be a general semigroup and let $x\in S$. Recall that $L_{x}$
denotes the $\Lc$-class of $x$.
\begin{lem}
The set $L_{x}$ is a partial $S$-action according to 
\[
s\cdot x=\begin{cases}
sx & sx\in L_{x}\\
\text{undefined} & \text{otherwise}
\end{cases}
\]
\end{lem}

\begin{proof}
Note that $S^{1}x=\{s\in S\mid s\leq_{\Lc}x\}$ and $L_{x}^{<}=\{s\in S\mid s\leq_{\Lc}x\wedge s\notin L_{x}\}$
are both left ideals of $S$. Then $L_{x}=S^{1}x\setminus L_{x}^{<}$
is a difference of two left ideals and therefore a partial left $S$-action.
\end{proof}
Let $\alpha\in S$ and $e,f\in E(S)$ such that $e\,\Rc\,\alpha\,\Lc\,f$.
Set $\beta$ to be the inverse of $\alpha$ such that $\alpha\beta=e$
and $\beta\alpha=f$ (we prefer to use $\alpha,\beta$ rather than
$a,b$ for elements of $S$ that are used for defining functions by
right multiplication). According to Green's lemma \cite[Lemma 2.2.1]{Howie1995}
the function $\rho_{\alpha}:L_{e}\to L_{f}$ defined by $\rho_{\alpha}(x)=x\alpha$
is a well-defined bijection whose inverse is $\rho_{\beta}:L_{f}\to L_{e}$.
\begin{lem}
\label{lem:Greens_lemma}The function $\rho_{\alpha}$ is a homomorphism
of partial $S$-actions.
\end{lem}

\begin{proof}
Let $x\in L_{e}$ and $s\in S$. If $sx\in L_{e}$ and $sx\alpha\in L_{f}$
it is clear that 
\[
\rho_{\alpha}(sx)=sx\alpha=s\rho_{\alpha}(x).
\]
It is left to show that 
\[
sx\in L_{e}\iff sx\alpha\in L_{f}.
\]
If $sx\in L_{e}$ then $\rho_{\alpha}(sx)=sx\alpha\in L_{f}$ because
$\rho_{\alpha}$ is well-defined. In the other direction, if $sx\alpha\in L_{f}$
then 
\[
\rho_{\beta}(sx\alpha)=sx\alpha\beta=sx\in L_{e}.
\]
\end{proof}
We now show that the generalized right ample identity is equivalent
to a similar property for the generalized Green's relation $\Lt$.
Fix $S$ to be a reduced $E$-Fountain semigroup which satisfies the
congruence condition. We denote the $\Lt$ -class of $e\in E$ by
$\Lt(e)$. It is proved in \cite[Lemma 5.1]{Margolis2021} that for
every $e\in E$, the set $\Lt(e)$ is a partial $S$-action according
to
\[
s\cdot x=\begin{cases}
sx & sx\in\Lt(e)\\
\text{undefined} & \text{otherwise}
\end{cases}
\]

for $s\in S$ and $x\in\Lt(e)$. 

For every $\alpha\in S$ we can define a function $r_{\alpha}:\Lt(\alpha^{+})\to\Lt(\alpha^{\ast})$
by $r_{\alpha}(x)=x\alpha$. Indeed, if $x\in\Lt(\alpha^{+})$ then
\[
(r_{\alpha}(x))^{\ast}=(x\alpha)^{\ast}=(x^{\ast}\alpha)^{\ast}=(\alpha^{+}\alpha)^{\ast}=\alpha^{\ast}
\]
so $r_{\alpha}(x)\in\Lt(\alpha^{\ast})$ hence $r_{\alpha}$ is a
well-defined function. 

In general, $r_{\alpha}$ is not a homomorphism of partial $S$-actions
but it is if the generalized right ample identity holds.
\begin{thm}
\label{thm:GRA_thm}Let $S$ be a reduced $E$-Fountain semigroup
which satisfies the congruence condition. The semigroup $S$ satisfies
the generalized right ample condition if and only if $r_{\alpha}$
is a homomorphism of partial $S$-actions for every $\alpha\in S$.
\end{thm}

\begin{proof}
First assume $S$ satisfies the generalized right ample condition.
Let $\text{\mbox{\ensuremath{x\in\Lt(\alpha^{+})}}}$. If $\mbox{\ensuremath{sx\in\Lt(\alpha^{+})}}$
and $sx\alpha\in\Lt(\alpha^{\ast})$ it is obvious that 
\[
r_{\alpha}(sx)=sx\alpha=sr_{\alpha}(x).
\]
It is left to show that 
\[
sx\in\Lt(\alpha^{+})\iff sx\alpha\in\Lt(\alpha^{\ast}).
\]
If $sx\in\Lt(\alpha^{+})$ then $(sx)^{\ast}=\alpha^{+}$ and, 
\[
(sx\alpha)^{\ast}=((sx)^{\ast}\alpha)^{\ast}=(\alpha^{+}\alpha)^{\ast}=\alpha^{\ast}
\]
so $sx\alpha\in\Lt(\alpha^{\ast})$ as required. In the other direction,
assume $sx\alpha\in\Lt(\alpha^{\ast})$ so $(sx\alpha)^{\ast}=\alpha^{\ast}$.
Choose $a=\alpha$ and $e=(sx)^{\ast}$. The generalized right ample
identity implies 
\[
((sx)^{\ast}(\alpha((sx)^{\ast}\alpha)^{\ast})^{+})^{\ast}=(\alpha((sx)^{\ast}\alpha)^{\ast})^{+}.
\]
The congruence condition implies $(sx\alpha)^{\ast}=((sx)^{\ast}\alpha)^{\ast}$
so we obtain 
\[
((sx)^{\ast}(\alpha(sx\alpha)^{\ast})^{+})^{\ast}=(\alpha(sx\alpha)^{\ast})^{+}.
\]
By assumption, $(sx\alpha)^{\ast}=\alpha^{\ast}$ so 
\[
((sx)^{\ast}(\alpha\alpha^{\ast})^{+})^{\ast}=(\alpha\alpha^{\ast})^{+}
\]
hence 
\[
((sx)^{\ast}\alpha^{+})^{\ast}=\alpha^{+}
\]
which implies
\[
(sx\alpha^{+})^{\ast}=\alpha^{+}
\]
by another use of the congruence condition. Finally, $\alpha^{+}=x^{\ast}$
because $\mbox{\ensuremath{x\in\Lt(\alpha^{+})}}$ so $\alpha^{+}=(sxx^{\ast})^{\ast}=(sx)^{\ast}$
hence $sx\in\Lt(\alpha^{+})$ as required. This finishes the ``only
if'' part. 

For the other direction assume $r_{\alpha}$ is a homomorphism of
partial $S$-actions for every $\alpha\in S$, and choose $e\in E$
and $a\in S$. We need to prove 
\[
(e(a(ea)^{\ast})^{+})^{\ast}=(a(ea)^{\ast})^{+}.
\]
First note that $eaa^{\ast}=ea$ hence $(ea)^{\ast}\leq a^{\ast}$.
This implies 
\[
(a(ea)^{\ast})^{\ast}=(a^{\ast}(ea)^{\ast})^{\ast}=((ea)^{\ast})^{\ast}=(ea)^{\ast}.
\]
Now, consider the function
\[
r_{a(ea)^{\ast}}:\Lt((a(ea)^{\ast})^{+})\to\Lt((a(ea)^{\ast})^{\ast})=\Lt((ea)^{\ast}).
\]
Note that 
\[
er_{a(ea)^{\ast}}((a(ea)^{\ast})^{+})=e(a(ea)^{\ast})^{+}a(ea)^{\ast}=ea(ea)^{\ast}=ea\in\Lt((ea)^{\ast}).
\]
The fact that $r_{a(ea)^{\ast}}$ is a homomorphism of partial $S$-actions
implies that $e(a(ea)^{\ast})^{+}\in\Lt((a(ea)^{\ast})^{+})$ hence
\[
(e(a(ea)^{\ast})^{+})^{\ast}=(a(ea)^{\ast})^{+}
\]
as required.
\end{proof}
For future reference we state also the dual result.
\begin{cor}
Let $S$ be a reduced $E$-Fountain semigroup which satisfies the
congruence condition. The semigroup $S$ satisfies the generalized
left ample identity
\[
(((ae)^{+}a)^{\ast}e)^{+}=((ae)^{+}a)^{\ast}
\]
(for every $a\in S$, $e\in E$) if and only if the function $l_{\alpha}:\Rt(\alpha^{\ast})\to\Rt(\alpha^{+})$
defined by $l_{\alpha}(x)=\alpha x$ is a homomorphism of partial
right $S$-actions for every $\alpha\in S$.
\end{cor}

It turns out that every homomorphism of partial $S$-actions between
sets of the form $\Lt(e)$ is $r_{\alpha}$ for a certain $\alpha\in S$
as we now show.
\begin{lem}
\label{lem:every_action_hom__is_right_mult}Let $S$ be a reduced
$E$-Fountain semigroup. Let $e,f\in E$ and assume that $\mbox{\ensuremath{F:\Lt(e)\to\Lt(f)}}$
is a homomorphism of partial left $S$-actions. Then $F=r_{\alpha}$
for $\alpha\in S$ with $\alpha^{+}=e$ and $\alpha^{\ast}=f$.
\end{lem}

\begin{proof}
Set $\alpha=F(e)$. It is clear that $\alpha^{\ast}=f$. To show $\alpha^{+}=e$
first observe that 
\[
\alpha=F(e)=F(ee)=eF(e)=e\alpha
\]
 so $\alpha^{+}\leq e$. Next, 
\[
\alpha=\alpha^{+}\alpha=\alpha^{+}F(e).
\]
The fact that $F$ is a homomorphism of partial $S$-actions implies
that
\[
\alpha^{+}e\in\Lt(e)
\]
and $\alpha^{+}\leq e$ implies $\alpha^{+}=\alpha^{+}e$. Therefore,
$\alpha^{+}\in\Lt(e)$ so $\alpha^{+}=e$ by the uniqueness of idempotents
from $E$ in an $\Lt$-class. Finally, for every $x\in\Lt(e)$ we
have $x^{\ast}=e$ so
\[
F(x)=F(xe)=xF(e)=x\alpha=r_{\alpha}(x)
\]
therefore, $F=r_{\alpha}$ as required.
\end{proof}
As a corollary we obtain a concrete interpretation of the associated
category $\mathcal{C}(S)$. Define $\mathcal{D}(S)$ to be a category
as follows: The objects of $\mathcal{D}(S)$ are sets of the form
$\Lt(e)$ for $e\in E$ and the hom-set $\mathcal{D}(S)(e,f)$ is
the set of all homomorphisms of partial left $S$-actions $F:\Lt(e)\to\Lt(f)$.
\begin{cor}
\label{cor:Category_isomomrphis}Let $S$ be a reduced E-Fountain
semigroup which satisfies the congruence condition and the generalized
right ample identity. The category $\mathcal{C}(S)^{\op}$ is isomorphic
to the category $\mathcal{D}(S)$.
\end{cor}

\begin{proof}
Define $\mathcal{F}:\mathcal{C}(S)\to\mathcal{D}(S)$ on objects by
$\mathcal{F}(e)=\Lt(e)$ and on morphisms by $\mathcal{F}(C(\alpha))=r_{\alpha}$.
First observe that $\mathcal{F}(C(e))=r_{e}$ is the identity function
on $\Lt(e)$ for every $e\in E$. Note also that the domain $C(\alpha)$
is $\alpha^{\ast}$ and its range is $\alpha^{+}$ while $r_{\alpha}:\Lt(\alpha^{+})\to\Lt(\alpha^{\ast}).$
In particular, if $\alpha^{+}=\beta^{\ast}$ for $\alpha,\beta\in S$
then
\[
r_{\alpha}r_{\beta}=r_{\beta\alpha}
\]
-composing functions right to left. Moreover, from the definition
of $\mathcal{C}(S)$ we have 
\[
C(\beta)C(\alpha)=C(\beta\alpha)
\]
 so 
\[
\mathcal{F}(C(\beta)C(\alpha))=\mathcal{F}(C(\beta\alpha))=r_{\beta\alpha}=r_{\alpha}r_{\beta}=\mathcal{F}(C(\alpha))\mathcal{F}(C(\beta))
\]
and $\mathcal{F}$ is indeed a contravariant functor. The functor
$\mathcal{F}$ is clearly a bijection on objects. To see that $\mathcal{F}$
is injective on hom-sets, consider two morphisms $C(\alpha),C(\alpha^{\prime})$
with domain $\alpha^{\ast}$ and range $\alpha^{+}$ such that $r_{\alpha}=\mathcal{F}(C(\alpha))=\mathcal{F}(C(\alpha^{\prime}))=r_{\alpha^{\prime}}$.
Then 
\[
\alpha=\alpha^{+}\alpha=r_{\alpha}(\alpha^{+})=r_{\alpha^{\prime}}(\alpha^{+})=\alpha^{+}\alpha^{\prime}=\alpha^{\prime}
\]
hence $C(\alpha)=C(\alpha^{\prime}).$ Finally, \lemref{every_action_hom__is_right_mult}
shows that $\mathcal{F}$ is onto on hom-sets and therefore it is
an isomorphism between $\mathcal{C}(S)^{\op}$ and $\mathcal{D}(S)$. 
\end{proof}

\subsection{\label{subsec:Modules}The category $\mathcal{C}(S)$ as a discrete
Peirce decomposition }

Let $S$ be a semigroup and let $\Bbbk$ be a field. For every set
$X$ we denote by $\Bbbk X$ the $\Bbbk$-vector space of all formal
finite linear combinations of elements of $X$
\[
\text{\ensuremath{\Bbbk X=}}\{k_{1}x_{1}+\ldots+k_{n}x_{n}\mid k_{i}\in\Bbbk,\quad x_{i}\in X\}.
\]
Assume $X$ is a partial $S$-action. Here it will be convenient to
denote this partial action by $s\bullet x$ for $s\in S$ and $x\in X$.
The $\Bbbk$-vector space $\Bbbk X$ has the structure of a $\Bbbk S$-module
according to 
\[
s\cdot x=\begin{cases}
s\bullet x & s\bullet x\text{ is defined}\\
0 & s\bullet x\text{ is undefined}
\end{cases}
\]

for $s\in S$ and $x\in X$, and extending linearly.  Let $X,Y$ be
two partial $S$-actions. It is clear that every function $f:X\to Y$
can be extended into a linear transformation $f:\Bbbk X\to\Bbbk Y$
which we denote also by $f$ for the sake of simplicity. It is a routine
to verify that $f:X\to Y$ is a homomorphism of partial $S$-actions
if and only if the extended function $f:\Bbbk X\to\Bbbk Y$ is a homomorphism
of $\Bbbk S$-modules. 

Now we specify to our case. Fix $S$ to be a reduced $E$-Fountain
semigroup which satisfies the congruence condition. The $\Bbbk$-vector
space $\Bbbk\Lt(e)$ has the structure of a $\Bbbk S$-module according
to 
\[
s\cdot x=\begin{cases}
sx & sx\in\Lt(e)\\
0 & sx\notin\Lt(e)
\end{cases}
\]
for $s\in S$ and $x\in\Lt(e)$, and extending linearly. We can extend
$r_{\alpha}$ to a linear transformation
\[
r_{\alpha}:\Bbbk\Lt(\alpha^{+})\to\Bbbk\Lt(\alpha^{\ast})
\]
which is a homomorphism of $\Bbbk S$-modules if and only if $r_{\alpha}:\Lt(\alpha^{+})\to\Lt(\alpha^{\ast})$
is a homomorphism of partial $S$-actions. Therefore, \thmref{GRA_thm}
implies the following corollary.
\begin{cor}
\label{cor:GRA_equivalent_KS_modules}Let $S$ be a reduced $E$-Fountain
semigroup which satisfies the congruence condition. The semigroup
$S$ satisfies the generalized right ample condition if and only if
$\mbox{\ensuremath{r_{\alpha}:\Bbbk\Lt(\alpha^{+})\to\Bbbk\Lt(\alpha^{\ast})}}$
is a $\Bbbk S$-module homomorphism for every $\alpha\in S$.
\end{cor}

Let $S$ be a reduced $E$-Fountain semigroup which satisfies the
congruence and the generalized right ample conditions. In view of
\corref{Category_isomomrphis} and \corref{GRA_equivalent_KS_modules}
we can identify the category $\mathcal{C}(S)^{\op}$ with a category
whose objects are left modules of the form $\Bbbk\Lt(e)$ and whose
morphisms are certain $\Bbbk S$-module homomorphisms between them,
but it is not true that every $\Bbbk S$-module homomorphism $F:\Bbbk\Lt(e)\to\Bbbk\Lt(f)$
is of the form $F=r_{\alpha}$ where $\alpha^{+}=e$ and $\alpha^{\ast}=f$.
However, we will see later that under certain conditions, every $\Bbbk S$-module
homomorphism is a linear combination of functions of this form.

From now we fix $S$ to be a \emph{finite} reduced $E$-Fountain semigroup
which satisfies the congruence condition. Some of the results in the
sequel are true also for infinite semigroups with weaker finiteness
conditions but for simplicity we prefer to deal only with the finite
case. Our goal in this section is to obtain a version of \corref{Category_isomomrphis}
for a linear category of $\Bbbk S$-modules. 

Recall that the relation $\trianglelefteq_{l}$ is defined on $S$
by the rule that $a\trianglelefteq_{l}b$ if and only if $a=be$ for
some $e\in E$. It is proved in \cite[Lemma 3.5]{Stein2021} that
$a\trianglelefteq_{l}b$ if and only if $a=ba^{\ast}$ and therefore
if $c\trianglelefteq_{l}b$ with $c^{\ast}=e$ then $c=be$.

From now on we assume that $S$ satisfies the generalized right ample
identity and $\trianglelefteq_{l}$ is contained in a partial order
so we can use the isomorphism $\Bbbk S\simeq\Bbbk\mathcal{C}(S)$
from \thmref{iso_theorem} freely. For instance, we can assume that
$\Bbbk S$ has a unit element because the category algebra $\Bbbk\mathcal{C}(S)$
is unital.

\thmref{iso_theorem} allows us to obtain some information on modules
of $\Bbbk S$. It is clear that if $M$ is a left $\Bbbk\mathcal{C}(S)$-module
it can also be viewed as a $\Bbbk S$-module according to 
\[
s\star m=\varphi(s)\cdot m=\sum_{t\trianglelefteq_{l}s}C(t)\cdot m
\]

for $s\in S$ and $m\in M$, where $\varphi$ is as defined in \thmref{iso_theorem}.
For every $e\in E$, the element $C(e)\in\Bbbk\mathcal{C}(S)$ is
an idempotent so $\Bbbk\mathcal{C}(S)\cdot C(e)$ is a projective
$\Bbbk\mathcal{C}(S)$-module. Therefore it is also a projective $\Bbbk S$-module
according to the linear extension of 
\[
s\star C(x)=\varphi(s)\cdot C(x)=\sum_{t\trianglelefteq_{l}s}C(t)\cdot C(x)
\]
where $s\in S$ and $C(x)$ is a morphism in $\mathcal{C}(S)$ whose
domain is $e$. Note that for every $e\in E$ both the domain and
range of $C(e)$ are the object $e$ so $C(e)C(f)=0$ for distinct
$e,f\in E$. Since ${\displaystyle \sum_{e\in E}C(e)}$ is the unit
element of $\Bbbk\mathcal{C}(S)$ we deduce that $\{C(e)\mid e\in E\}$
is a complete set of orthogonal idempotents for $\Bbbk\mathcal{C}(S)$
so we have a decomposition 
\[
\Bbbk S\simeq\Bbbk\mathcal{C}(S)\simeq\bigoplus_{e\in E}\Bbbk\mathcal{C}(S)\cdot C(e)
\]
as $\Bbbk S$-modules.

The next theorem is a generalization of \cite[Proposition 6.6]{Margolis2021}.
\begin{thm}
\label{thm:Projective_modules}Let $S$ be a finite reduced $E$-Fountain
semigroup which satisfies the congruence condition and the generalized
right ample identity. Assume that $\trianglelefteq_{l}$ is contained
in a partial order. For every $e\in E$ there is an isomorphism $\text{\mbox{\ensuremath{\Bbbk\Lt(e)\simeq\Bbbk\mathcal{C}(S)\cdot C(e)}}}$
of $\Bbbk S$-modules. In particular, $\Bbbk\Lt(e)$ is projective
for every $e\in E$ and $\mbox{\ensuremath{\Bbbk S\simeq{\displaystyle \bigoplus_{e\in E}}\Bbbk\Lt(e)}}$
as $\Bbbk S$-modules.
\end{thm}

\begin{proof}
The second statement follows from the first. To prove the first statement,
define 
\[
\mbox{\ensuremath{\Phi:\Bbbk\Lt(e)\to\Bbbk\mathcal{C}(S)\cdot C(e)}}
\]
 on basis elements by 
\[
\Phi(x)=C(x).
\]
Note that $x\in\Lt(e)$ if and only if $x^{\ast}=e$ which means that
the domain of $C(x)$ is $e$. Therefore, $\Phi$ is a well-defined
isomorphism of $\Bbbk$-vector spaces. It is left to show that for
every $s\in S$
\[
\Phi(s\cdot x)=s\star\Phi(x).
\]
\begin{casenv}
\item If $sx\in\Lt(e)$ then $(s^{\ast}x)^{\ast}=(sx)^{\ast}=x^{\ast}.$
Choose $a=x$, $e=s^{\ast}$ so the generalized right ample condition
implies that
\[
\left(s^{\ast}\left(x\left(s^{\ast}x\right)^{\ast}\right)^{+}\right)^{\ast}=\left(x\left(s^{\ast}x\right)^{\ast}\right)^{+}.
\]
Plugging $(s^{\ast}x)^{\ast}=x^{\ast}$ we obtain 
\[
\left(s^{\ast}\left(xx^{\ast}\right)^{+}\right)^{\ast}=\left(xx^{\ast}\right)^{+}
\]
so
\[
(s^{\ast}x^{+})^{\ast}=x^{+}
\]
and the congruence condition implies 
\[
(sx^{+})^{\ast}=(s^{\ast}x^{+})^{\ast}=x^{+}.
\]
Now 
\[
s\star\Phi(x)=\sum_{t\trianglelefteq_{l}s}C(t)\cdot C(x).
\]
For every $t$ in this summation, $t\trianglelefteq_{l}s$ implies
$t=st^{\ast}$. If $t\neq sx^{+}$ then $t^{\ast}\neq x^{+}$ so $C(t)\cdot C(x)=0$.
Therefore,
\[
s\star\Phi(x)=\sum_{t\trianglelefteq_{l}s}C(t)\cdot C(x)=C(sx^{+})\cdot C(x).
\]
We have already seen that $(sx^{+})^{\ast}=x^{+}$ so we obtain 
\[
C(sx^{+})\cdot C(x)=C(sx^{+}x)=C(sx)=\Phi(sx)=\Phi(s\cdot x).
\]
This proves that
\[
s\star\Phi(x)=\Phi(s\cdot x)
\]
 as required.
\item If $sx\notin\Lt(e)$ then $s\cdot x=0$. This happens when $(sx)^{\ast}\neq x^{\ast}$.
Assume that $s\star\Phi(x)\neq0$. Since $s\star\Phi(x)={\displaystyle \sum_{t\trianglelefteq_{l}s}}C(t)\cdot C(x)$
there exists a $t$ such that $t\trianglelefteq_{l}s$ and $t^{\ast}=x^{+}$.
This implies that $t=sx^{+}$ so $(sx^{+})^{\ast}=x^{+}$. However,
\[
(sx)^{\ast}=(sx^{+}x)^{\ast}=((sx^{+})^{\ast}x)^{\ast}=(x^{+}x)^{\ast}=x^{\ast}
\]
which is a contradiction. Therefore, 
\[
s\star\Phi(x)=0=\Phi(s\cdot x)
\]
also in this case. 
\end{casenv}
\end{proof}
\begin{cor}
\label{cor:Basis_of_hom_sets}With the same assumptions as in \thmref{Projective_modules}.
For every $e,f\in E$, the set 
\[
B=\{r_{\alpha}\mid\alpha\in S,\quad a^{+}=e,\quad\alpha^{\ast}=f\}
\]
 is a basis for the $\Bbbk$-vector space $\Hom_{\Bbbk S}(\Bbbk\Lt(e),\Bbbk\Lt(f))$.
\end{cor}

\begin{proof}
\corref{GRA_equivalent_KS_modules} implies that $r_{\alpha}\in\Hom_{\Bbbk}(\Bbbk\Lt(e),\Bbbk\Lt(f))$
if $\alpha^{+}=e$ and $\alpha^{\ast}=f$. Let $B=\{r_{\alpha_{1}},\ldots,r_{\alpha_{l}}\}$
where all the listed elements are distinct. To show that $B$ is linearly
independent, assume 
\[
k_{1}r_{\alpha_{1}}+\ldots+k_{l}r_{\alpha_{l}}=0
\]
 so in particular 
\[
k_{1}r_{\alpha_{1}}(e)+\ldots+k_{l}r_{\alpha_{l}}(e)=0.
\]
Since $r_{\alpha_{i}}(e)=e\alpha_{i}=\alpha_{i}$, we obtain
\[
k_{1}\alpha_{1}+\ldots+k_{l}\alpha_{l}=0
\]
which implies $k_{1}=\ldots=k_{l}=0$ since $\{\alpha_{1},\ldots,\alpha_{l}\}$
is a linearly independent set in $\Bbbk\Lt(f)$. Now, according to
\thmref{Projective_modules} we know that 
\begin{align*}
\dim_{\Bbbk}\Hom_{\Bbbk S}(\Bbbk\Lt(e),\Bbbk\Lt(f)) & =\dim_{\Bbbk}\Hom_{\Bbbk S}(\Bbbk\mathcal{C}(S)\cdot C(e),\Bbbk\mathcal{C}(S)\cdot C(f))\\
 & =\dim_{\Bbbk}C(e)\Bbbk\mathcal{C}(S)C(f)
\end{align*}
and this is precisely the number of $\alpha\in S$ such that $\alpha^{+}=e$
and $\alpha^{\ast}=f$. In conclusion, $\{r_{\alpha}\mid\alpha\in S,\quad a^{+}=e,\quad\alpha^{\ast}=f\}$
is a linearly independent set whose size is the dimension of $\Hom_{\Bbbk S}(\Bbbk\Lt(e),\Bbbk\Lt(f))$
and therefore it is a basis.
\end{proof}
Set $\mathcal{L}(\Bbbk S)$ to be the linear category (identified
with a Peirce decomposition of $\Bbbk S$) whose set of objects are
the projective modules of the form $\Bbbk\Lt(e)$ for $e\in E$ and
\[
\mathcal{L}(\Bbbk S)(\Bbbk\Lt(e),\Bbbk\Lt(f))=\Hom_{\Bbbk S}(\Bbbk\Lt(e),\Bbbk\Lt(f))
\]
 for every $e,f\in E$.
\begin{thm}
\label{thm:Pierce_decomposition}With the same assumptions as in \thmref{Projective_modules}.
There is an isomorphism of linear categories $\Bbbk[\mathcal{C}(S)]^{\op}\simeq\mathcal{L}(\Bbbk S)$. 
\end{thm}

\begin{proof}
Define a functor of $\Bbbk$-linear categories $\mathcal{F}:\Bbbk[\mathcal{C}(S)]\to\mathcal{L}(\Bbbk S)$
in the following way. On objects, $\mathcal{F}(e)=\Bbbk\Lt(e)$ and
$\mathcal{F}(C(\alpha))=r_{\alpha}$ defines $\mathcal{F}$ on the
bases of the hom-sets. The argument in \corref{Category_isomomrphis}
shows that $\mathcal{F}$ is a contravariant functor of $\Bbbk$-linear
categories. It is clearly a bijection on objects and \corref{Basis_of_hom_sets}
shows that $\mathcal{F}$ is an isomorphism of hom-sets.
\end{proof}
\thmref{Pierce_decomposition} shows that up to taking the opposite
category, $\Bbbk[\mathcal{C}(S)]$ is isomorphic to a Peirce decomposition
of $\Bbbk S$. The key fact is in the proof of \corref{Basis_of_hom_sets}.
If the generalized right ample identity holds, all homomorphisms between
the projective modules $\Bbbk\Lt(e)$ and $\Bbbk\Lt(f)$ are linear
combinations of functions of the form $r_{\alpha}$ and only for elements
$\alpha\in S$ which satisfy $\alpha^{+}=e$ and $\alpha^{\ast}=f$.
Therefore $\mathcal{L}(\Bbbk S)$ is (the opposite of) the linearization
of $\mathcal{C}(S)$ itself. In fact, we can think of the associated
category $\mathcal{C}(S)$ as being a ``discrete'' Peirce decomposition
of the semigroup $S$ itself.

\section{\label{sec:Examples}Examples}

A good way to show that the generalized right ample condition is a
natural property is to give some natural examples of semigroups which
satisfy it. If the set of idempotents $E$ is a subband of $S$, it
is shown in \cite{Stein2021} that the generalized right ample identity
reduces to the well-studied ``standard'' right ample identity. So
we are interested in examples where $E$ is not a subband of $S$.
One important example - the Catalan monoid - was already given in
\cite{Stein2021} (and will be recalled in the sequel as well).

In this section we give two additional examples of such semigroups.

\subsection{Linear operators on a Hilbert space \protect\footnote{This example is part of the ArXiv version of \cite{Stein2021} but
was omitted from the final paper due to referee's request.}}

Let $H$ be a Hilbert space and let $B(H)$ be the algebra of all
bounded linear operators on $H$ (this is one of the main examples
of a Rickart $\ast$-ring and a Baer $\ast$-ring, see \cite{berberian1988baer,Stokes2015}).
For every closed subspace $U\subseteq H$ we associate the orthogonal
projection $P_{U}\in B(H)$ onto it. Recall that $P_{U}$ is an idempotent
and $\im(P_{U})=U=\ker(P_{U})^{\perp}$ (where $\im T$ is the image
of the linear operator $T$, $\ker(T)=\{h\in H\mid T(h)=0\}$ is the
kernel of the linear operator $T$ and $V^{\perp}$ is the orthogonal
complement of $V$). It is easy to check that 
\[
P_{U}P_{V}=P_{U}\iff P_{V}P_{U}=P_{U}\iff U\subseteq V.
\]
Now, consider $B(H)$ as a multiplicative monoid. This is an example
of a Baer $\ast$-semigroup \cite{Foulis1960}. Let $T\in B(H)$ and
consider the restriction to $(\ker T)^{+}$
\[
T_{r}:(\ker T)^{+}\to\im T.
\]
Since every $h\in H$ can be written uniquely as $h=h_{1}+h_{2}$
where $h_{1}\in(\ker T)^{+}$ and $h_{2}\in\ker T$ we deduce that
$T_{r}$ is an isomorphism and $T$ can be derived from $T_{r}$ by
\[
T(h)=T_{r}(h_{1}).
\]
Now consider the inverse $T_{r}^{-1}:\im T\to(\ker T)^{+}$ and extend
it to a linear transformation $\mbox{\ensuremath{R:H\to H}}$, so
$\im R=(\ker T)^{+}$ and $\ker R=(\im T)^{+}$. It is possible to
verify that $T_{r}^{-1}$ is bounded (\cite[Corollary 2.12]{Rudin1991})
and therefore $R$ is bounded as well. It is also clear that $TRT=T$
so $B(H)$ is a regular semigroup. Being a regular subsemigroup of
the semigroup $L(H)$ of all linear operators on $H$, it inherits
the $\Rc$ and $\Lc$ equivalences from the semigroup $L(H)$ (see
and \cite[Proposition 2.4.2]{Howie1995}) and the $\Rc$ and $\Lc$
relations on $L(H)$ are well-known \cite[Excercise 19 on page 63]{Howie1995}.
Therefore we obtain:
\begin{lem}
For every $T_{1},T_{2}\in B(H)$ we have 
\[
T_{1}\,\Lc\,T_{2}\iff\ker(T_{1})=\ker(T_{2}),\qquad T_{1}\,\Rc\,T_{2}\iff\im(T_{1})=\im(T_{2}).
\]
\end{lem}

In particular, 
\[
T\,\Lc\,P_{(\ker T)^{\perp}},\quad T\,\Rc\,P_{\im T}
\]
for every $T\in B(H)$. Now, if we set $E=\{P_{U}\mid U\text{ is a closed subspace of }H\}$
then $\Lc\subseteq\Lt$ and $\Rc\subseteq\Rt$ implies 
\[
T\,\Lt\,P_{(\ker T)^{\perp}},\quad T\,\Rt\,P_{\im T}.
\]
Therefore, $B(H)$ is a reduced $E$-Fountain semigroup and note that
in general $P_{U}P_{V}\neq P_{V}P_{U}$ so $E$ is not a subband of
$B(H)$. In fact, since every $T\in B(H)$ is $\Lc$-equivalent to
an idempotent and $\Rc$-equivalent to an idempotent of $E$ we deduce
that $\Lt=\Lc$ and $\Rt=\Rc$ so this is a very special case of a
reduced $E$-Fountain semigroup. This allows us to prove the rest
of the properties easily. The relation $\Lc$ is a right congruence
and $\Rc$ is a left congruence (\cite[Proposition 2.1.2]{Howie1995})
so we deduce immediately that $B(H)$ satisfies the congruence condition.
Furthermore, \lemref{Greens_lemma} implies that $B(H)$ satisfies
both the right and left generalized ample identities.
\begin{rem}
The reader might be disappointed by an example where $\Lt=\Lc$ and
$\Rt=\Rc$ because in this case the congruence and the generalized
ample conditions reduce to known facts on Green's relations (and this
is also the situation in the next example in this section). However,
our conditions are equational identities so it is possible to find
more examples by considering subsemigroups which contain the set $E$.
For instance, the subsemigroup $\langle E\rangle$ of $B(H)$ generated
by the projections $P_{U}$ is another reduced $E$-Fountain semigroup
which satisfies the congruence condition and the generalized ample
conditions. It is possible to check (but we will not give the details)
that in this semigroup the containments $\Lc\subset\Lt$ and $\Rc\subset\Rt$
are strict.
\end{rem}

\subsection{Order-preserving functions with a fixed point }

Let $\T_{n}$ be the monoid of all functions $f:[n]\to[n]$ (where
$[n]=\{1,\ldots,n\}$). A function is called \emph{order-preserving}
if $i_{1}\leq i_{2}$ implies $f(i_{1})\leq f(i_{2})$. Denote by
$\Op_{n}$ the monoid of all order-preserving functions $f:[n]\to[n]$.
We denote by $\OF_{n}$ the monoid of all order-preserving functions
$f:[n]\to[n]$ with the fixed point $f(n)=n$. 
\begin{example}
\label{exa:Example_OF2}The monoid $\OF_{3}$ consists of the following
$6$ elements:
\[
\left(\begin{array}{ccc}
1 & 2 & 3\\
1 & 2 & 3
\end{array}\right),\left(\begin{array}{ccc}
1 & 2 & 3\\
1 & 3 & 3
\end{array}\right),\left(\begin{array}{ccc}
1 & 2 & 3\\
2 & 3 & 3
\end{array}\right)
\]
\[
\left(\begin{array}{ccc}
1 & 2 & 3\\
1 & 1 & 3
\end{array}\right),\left(\begin{array}{ccc}
1 & 2 & 3\\
2 & 2 & 3
\end{array}\right),\left(\begin{array}{ccc}
1 & 2 & 3\\
3 & 3 & 3
\end{array}\right)
\]
\end{example}

Clearly, $\OF_{n}$ is a submonoid of $\Op_{n}$. Note the analogy
with the definition of the monoid $\PT_{n}$ of all partial functions
on an $n$-element set. The monoid $\PT_{n}$ can be identified with
the submonoid of $\T_{n+1}$ which consists of all functions $f:[n+1]\to[n+1]$
such that $f(n+1)=n+1$.

To the author's knowledge, the monoid $\OF_{n}$ has not been studied
before. In this section we will discuss some properties of this monoid
with emphasis on its structure as a reduced $E$-Fountain semigroup.

Let $f\in\OF_{n}$. Recall that the kernel of $f$ - $\ker(f)$ -
is the equivalence relation on its domain defined by $(i_{1},i_{2})\in\ker(f)$
if and only if $f(i_{1})=f(i_{2})$. Since $f$ is order-preserving,
the kernel classes of $f$ are \emph{intervals }- if $i_{1}\leq i_{2}\leq i_{3}$
and $(i_{1},i_{3})\in\ker(f)$ then $(i_{1},i_{2})\in\ker f$ also.

Let $K_{1},\ldots,K_{l+1}$ be the kernel classes of $f$ (where $0\leq l\leq n-1$).
Let $y_{i}=f(K_{i}$) and assume that the indices are arranged such
that $y_{1}<y_{2}<\cdots<y_{l+1}$. Choose $x_{i}$ to be the maximal
element of $K_{i}$. Since $f(n)=n$ it must be the case that $y_{l+1}=n$
and $x_{l+1}=n$. Now, set 
\[
X=\{x_{1},\ldots,x_{l}\},\quad Y=\{y_{1},\ldots,y_{l}\}
\]
so from every $f\in\OF_{n}$ we can extract two sets $X,Y\subseteq[n-1]$
such that $|X|=|Y|$. In the other direction, from two such sets $X=\{x_{1},\ldots,x_{l}\}$
and $\mbox{\ensuremath{Y=\{y_{1},\ldots,y_{l}\}}}$ we can retrieve
an $f$ by setting 
\begin{equation}
f(x)=\begin{cases}
y_{1} & x\leq x_{1}\\
y_{i} & x_{i-1}<x\leq x_{i}\quad(2\leq i\leq l)\\
n & x_{l}<x.
\end{cases}\label{eq:Def_of_OFn_element}
\end{equation}
As a conclusion, we observe that there is a one-to-one correspondence
between functions $f\in\OF_{n}$ and pairs of sets $X,Y\subseteq[n-1]$
such that $|X|=|Y|$. Note that $\im(f)=Y\cup\{n\}$ (where $\im(f)$
is the image of $f$) and $X$ contains the maximal elements of the
kernel classes of $f$ except $n$. From now on we denote by $f_{X,Y}$
the function associated with $X,Y\subseteq[n-1]$. For instance, the
functions from \exaref{Example_OF2} are
\[
f_{\{1,2\},\{1,2\}},f_{\{1\},\{1\}},f_{\{1\},\{2\}},f_{\{2\},\{1\}},f_{\{2\},\{2\}},f_{\varnothing,\varnothing}.
\]

\begin{cor}
The number of elements of $\OF_{n}$ is ${\displaystyle \sum_{k=0}^{n-1}\binom{n-1}{k}^{2}={2n-2 \choose n-1}}.$
\end{cor}

\begin{proof}
We sum over all pairs $X,Y\subseteq[n-1]$ with the same size. The
final equality follows from \cite[Example 1.1.17]{Stanley1997}.
\end{proof}
Another observation will be useful. A function $f:[n]\to[n]$ is called
\emph{order-increasing} if $i\leq f(i)$ for every $i\in[n]$. In
particular, if $f$ is order-increasing then $f(n)=n$. Let $\C_{n}$
be the submonoid of $\OF_{n}$ which consists of all order-preserving
and order-increasing functions $f:[n]\to[n]$. This monoid is called
the Catalan monoid because the size of $\C_{n}$ is the $n$-th Catalan
number \cite[Theorem 14.2.8]{Ganyushkin2009b}. It is well-known that
$\C_{n}$ is a $\Jc$-trivial monoid \cite[Proposition 17.17]{Steinberg2016}.
Given $X=\{x_{1},\ldots,x_{l}\}$ and $Y=\{y_{1},\ldots,y_{l}\}$
we say that $X\leq Y$ if $x_{i}\leq y_{i}$ for $1\leq i\leq l$.
It is easy to see that $f_{X,Y}\in\OF_{n}$ is order-increasing if
and only if $X\leq Y$ (see also \cite{Margolis2018A}). Therefore,
\[
\C_{n}=\{f_{X,Y}\in\OF_{n}\mid X\leq Y\}.
\]

In the rest of this section we consider several properties of the
monoid $\OF_{n}$ and its algebra.

\paragraph*{Idempotents}

Let $X=\{x_{1},\ldots,x_{l}\}$ and $Y=\{y_{1},\ldots,y_{l}\}$ be
two subsets of $[n-1]$. It is well-known (see \cite[Theorem 2.7.2]{Ganyushkin2009b})
that a function $f$ is an idempotent if $f(i)=i$ for every $i\in\im(f)$.
In view of (\ref{eq:Def_of_OFn_element}) and the fact that $\im(f_{X,Y})=Y\cup\{n\}$
it follows that $f_{X,Y}$ is an idempotent if and only if $y_{1}\leq x_{1}$
and $x_{i-1}<y_{i}\leq x_{i}$ for $2\leq i\leq l$.

\paragraph{Green's relations}

It is easy to verify that $f_{Y,Z}f_{X,Y}=f_{X,Z}$ so $f_{X,Y}f_{Y,X}f_{X,Y}=f_{X,Y}$
and therefore $\OF_{n}$ is a regular monoid. As a regular submonoid
of $\T_{n}$ it is easy to describe its $\Rc$ and $\Lc$ relations,
as explained in the previous example (see \cite[Proposition 2.4.2]{Howie1995}).
The $\Rc$ and $\Lc$ relations of $\T_{n}$ can be found in \cite[Excercise 16 on page 63]{Howie1995}.
\begin{lem}
\label{lem:R_L_Description}Let $f_{X,Y},f_{Z,W}\in\OF_{n}$, then
\begin{align*}
f_{X,Y}\,\Rc\,f_{Z,W} & \iff\im(f_{X,Y})=\im(f_{Z,W})\iff Y=W\\
f_{X,Y}\,\Lc\,f_{Z,W} & \iff\ker(f_{X,Y})=\ker(f_{Z,W})\iff X=Z.
\end{align*}
\end{lem}

It follows immediately that $f_{X,Y}\,\Hc\,f_{Z,W}\iff f_{X,Y}=f_{Z,W}$
so $\OF_{n}$ is an $\Hc$-trivial monoid (which is also clear from
the fact that $\OF_{n}$ is a regular submonoid of $\Op_{n}$). In
particular, it has only trivial subgroups.

Next, we turn to describe the $\Jc$-classes. Recall that the rank
of a function $f$ is the size of its image. For $f_{X,Y}\in\OF_{n}$
we have $\rank(f_{X,Y})=|X|+1=|Y|+1$.
\begin{lem}
Let $f_{X,Y},f_{Z,W}\in\OF_{n}$ then $f_{X,Y}\,\Jc\,f_{Z,W}\iff\rank(f_{X,Y})=\rank(f_{Z,W})$.
\end{lem}

\begin{proof}
Since $\OF_{n}$ is a finite monoid we know that $\Jc=\Rc\circ\Lc=\Lc\circ\Rc$
(see \cite[Proposition 2.1.4]{Howie1995}). Now \lemref{R_L_Description}
implies that 
\[
f_{X,Y}\,\Jc\,f_{Z,W}\iff f_{X,Y}\,\Rc\,f_{Z,Y}\,\Lc\,f_{Z,W}\iff|Z|=|Y|\iff\rank(f_{X,Y})=\rank(f_{Z,W}).
\]
\end{proof}
In conclusion, we obtain a neat description for the ``eggbox'' diagrams
of $\OF_{n}$ (see \cite[Section 2.2]{Howie1995} for the meaning
of an ``eggbox'' diagram). The monoid $\OF_{n}$ has $n$ $\Jc$-classes.
For every $0\leq k\leq n-1$ we can associate a $\Jc$-class $J_{k}$
of all functions $f_{X,Y}$ such that $|X|=|Y|=k$. Both the $\Rc$-classes
and the $\Lc$-classes in $J_{k}$ are indexed by subsets $Z\subseteq[n-1]$
with $|Z|=k$. For every $Z\subseteq[n-1]$ with $|Z|=k$ we have
an $\Rc$-class
\[
R_{Z}=\{f_{X,Z}\mid X\subseteq[n-1],\quad|X|=|Z|\}
\]
and an $\Lc$-class
\[
L_{Z}=\{f_{Z,Y}\mid Y\subseteq[n-1],\quad|Y|=|Z|\}.
\]

Every $\Hc$-class contains one element. Every $\Lc$-class ($\Rc$-class)
in $J_{k}$ contains ${n-1 \choose k}$ elements. The $\Jc$-class
$J_{k}$ contains ${n-1 \choose k}$ $\Lc$-classes ($\Rc$-classes)
and a total of ${n-1 \choose k}^{2}$ elements.

\paragraph{Reduced $E$-Fountain structure}

Set $E=\{f_{X,X}\mid X\subseteq[n-1]\}\subseteq E(\OF_{n})$. Note
that $E$ does not contain all the idempotents of $\OF_{n}$, but
it is precisely the set of idempotents of $\C_{n}$. According to
\lemref{R_L_Description}, we know that
\[
f_{Y,Y}\,\Rc\,f_{X,Y}\,\Lc\,f_{X,X}
\]
for every $X,Y\subseteq[n-1]$ with $|X|=|Y|$. Therefore, every $\Lc$-class
and every $\Rc$-class contains an element of $E$. Since $\Lc\subseteq\Lt$
and $\Rc\subseteq\Rt$, we deduce that every $\Lt$-class and every
$\Rt$-class contains an element of $E$ so $\OF_{n}$ is an $E$-Fountain
semigroup. Now, it is well-known (see \cite[Lemma 3.6]{Denton2010})
that the idempotents of a $\Jc$-trivial semigroup satisfy the property
that $ef=e$ if and only if $fe=e$. Since $E$ is the set of idempotents
of the $\Jc$-trivial semigroup $\C_{n}$ we obtain that $\OF_{n}$
is in fact a reduced $E$-Fountain semigroup. Note that 
\[
(f_{X,Y})^{\ast}=f_{X,X},\quad(f_{X,Y})^{+}=f_{Y,Y}
\]
and therefore
\[
f_{X,Y}\,\Lt\,f_{Z,W}\iff X=Z,\quad f_{X,Y}\,\Rt\,f_{Z,W}\iff Y=W.
\]
In particular we have 
\[
\Lt=\Lc,\quad\Rt=\Rc.
\]

As in the previous example, this immediately implies that $\OF_{n}$
satisfies the congruence condition (by \cite[Proposition 2.1.2]{Howie1995})
and both the right and left generalized ample conditions (by \lemref{Greens_lemma}).
The congruence condition implies that $\OF_{n}$ has an associated
category $\mathcal{C}(\OF_{n})$. The objects of $\mathcal{C}(\OF_{n})$
are in one-to one correspondence with sets $X\subseteq[n-1].$ Every
$f_{X,Y}\in\OF_{n}$ corresponds to a morphism from $X$ to $Y$ so
there exists a (unique) morphism from $X$ to $Y$ if and only if
$|X|=|Y|$. In other words, $\mathcal{C}(\OF_{n})$ represent the
equivalence relation $\sim$ defined on subsets of $[n-1]$ by $X\sim Y\iff|X|=|Y|$.
Two objects $X,Y\subseteq[n-1]$ are isomorphic if and only if $|X|=|Y|$
and in this case there is a unique morphism from $X$ to $Y$. In
particular, $\mathcal{C}(\OF_{n})$ is both a groupoid and a locally
trivial category - where a locally trivial category is a category
whose only endomorphisms (i.e., morphisms whose domain and range are
equal) are the identity morphisms. 

Since the Catalan monoid $\C_{n}$ is a submonoid of $\OF_{n}$ and
$E\subseteq\C_{n}$ it is also clear that $\C_{n}$ satisfies the
right and left generalized ample identities. This is a much neater
proof of this fact than the one given in \cite{Stein2021}.

\paragraph{The semigroup algebra}

Define a partial order $\preceq$ on $\OF_{n}$ by 
\[
\mbox{\ensuremath{f_{Z,W}\preceq f_{X,Y}\iff W\subsetneqq Y}}\text{ or }\mbox{\ensuremath{(W=Y\text{ and }Z\leq X)}}.
\]
Note that this is equivalent to
\[
\mbox{\ensuremath{f_{Z,W}\preceq f_{X,Y}\iff\im(f_{Z,W})\subsetneqq\im(f_{X,Y})}}\text{ or }\mbox{\ensuremath{(\im(f_{Z,W})=\im(f_{X,Y})\text{ and }Z\leq X)}}
\]
It is easy to verify that this is indeed a partial order (it is the
lexicographic order on the Cartesian product of $\subseteq$ and $\leq$).
Recall that the relation $\trianglelefteq_{l}$ is defined on a reduced
$E$-Fountain semigroup by the rule that $a\trianglelefteq_{l}b$
if and only if $a=be$ for some $e\in E$.
\begin{lem}
\label{lem:In_a_partial_order} In the semigroup $\OF_{n}$, there
is a containment of relations $\trianglelefteq_{l}\subseteq\preceq$.
\end{lem}

\begin{proof}
Let $f_{X,Y}\in\OF_{n}$ and $f_{Z,Z}\in E$. We need to show that
$f_{X,Y}f_{Z,Z}\preceq f_{X,Y}.$ First, note that $\im(f_{X,Y}f_{Z,Z})\subseteq\im(f_{X,Y})=Y\cup\{n\}$.
If $\im(f_{X,Y}f_{Z,Z})\subsetneqq Y\cup\{n\}$ then we are done.
It is left to consider the option that $\im(f_{X,Y}f_{Z,Z})=Y\cup\{n\}$.
In other words we assume that $f_{X,Y}f_{Z,Z}\,\Rc\,f_{X,Y}$. The
fact that $\Rc$ is a left congruence implies that 
\[
f_{X,X}f_{Z,Z}=f_{Y,X}f_{X,Y}f_{Z,Z}\,\Rc\,f_{Y,X}f_{X,Y}=f_{X,X}
\]
so
\[
\im(f_{X,X}f_{Z,Z})=X\cup\{n\}.
\]
Therefore, we can write 
\[
f_{X,X}f_{Z,Z}=f_{Z^{\prime},X}
\]
and observe that $f_{X,X},f_{Z,Z}\in\C_{n}$ so $f_{Z^{\prime},X}\in\C_{n}$
as well hence $Z^{\prime}\leq X$. Finally, 
\[
f_{X,Y}f_{Z,Z}=f_{X,Y}f_{X,X}f_{Z,Z}=f_{X,Y}f_{Z^{\prime},X}=f_{Z^{\prime},Y}.
\]
and therefore $f_{X,Y}f_{Z,Z}\preceq f_{X,Y}$ as required. 
\end{proof}
\thmref{iso_theorem} and \lemref{In_a_partial_order} implies that
for every field $\Bbbk$ we have an isomorphism of algebras 
\[
\Bbbk\OF_{n}\simeq\Bbbk\mathcal{C}(\OF_{n}).
\]

This isomorphism allows us to obtain a lot of information on the algebra
$\Bbbk\OF_{n}$. For instance we have the following immediate corollary:
\begin{lem}
For every field $\Bbbk$ the algebra $\Bbbk\OF_{n}$ is semisimple.
\end{lem}

\begin{proof}
It is well-known that the algebra $\Bbbk\mathcal{G}$ of a finite
groupoid $\mathcal{G}$ is semisimple if and only if the order of
every endomorphism group is invertible in $\Bbbk$ (\cite[Theorem 8.15]{Steinberg2016}).
In our case, the endomorphism groups of $\mathcal{C}(\OF_{n})$ are
trivial so $\Bbbk\mathcal{C}(\OF_{n})$ is semisimple for every field
$\Bbbk$ and hence also so is $\Bbbk\OF_{n}$.
\end{proof}
We remark that not many examples are known for non-inverse semigroups
with a semisimple algebra. Another notable example is the semigroup
of matrices over a finite field of appropriate characteristic (see
\cite[Section 5.6]{Steinberg2016}).

\paragraph{The monoid of order-preserving partial permutations}

Recall that the symmetric inverse monoid $\IS_{n}$ is the monoid
of all partial permutations on the set $[n]$. Our next observation
is related to the submonoid $\IO_{n}$ of all order-preserving partial
permutations (also denoted sometimes $\POI_{n}$- see \cite{Fernandes2001,Fernandes1997,Mazorchuk2003}).
In other words, $\IO_{n}$ consists of all partial permutations $\theta:[n]\to[n]$
such that $\theta(i)<\theta(j)$ if $i,j$ are in the domain of $\theta$
and $i<j$. It is clear that every $\theta\in\IO_{n}$ is completely
determined by its domain and image. In the other direction, given
$X,Y\subseteq[n]$ such that $|X|=|Y|$ there exists a unique $\theta_{X,Y}\in\IO_{n}$
whose domain is $X$ and image is $Y$. Therefore, elements of $\IO_{n}$
can also be indexed by pairs of sets $X,Y\subseteq[n]$ such that
$|X|=|Y|$. It is well-known that $\IO_{n}$ is an inverse semigroup
so in particular, it is a reduced $E$-Fountain semigroup for $E=E(\IO_{n})$
which satisfies the congruence condition. Moreover, its idempotents
commute and it satisfies the left and right ample conditions. Consider
the associated category $\mathcal{C}(\IO_{n})$ (which is in fact
an inductive groupoid - see \cite[Chapter 4]{Lawson1998}). It is
easy to deduce that the categories $\mathcal{C}(\IO_{n})$ and $\mathcal{C}(\OF_{n+1})$
are isomorphic. Indeed, the idempotents of $\IO_{n}$ are the partial
permutations of the form $\theta_{X,X}$ - where the domain equals
the image. In other words these are precisely the partial identities.
So objects of $\mathcal{C}(\IO_{n})$ are in one to one correspondence
with sets $X\subseteq[n]$. For every $\theta\in\IO_{n}$ there exists
a unique morphism $C(\theta)$ from the domain of $\theta$ to its
image. So the isomorphism between $\mathcal{C}(\IO_{n})$ and $\mathcal{C}(\OF_{n+1})$
is clear. According to \thmref{iso_theorem} (which for the case of
inverse semigroups is precisely \cite[Theorem 4.2]{Steinberg2006})
we know that $\Bbbk\mathcal{C}(\IO_{n})\simeq\Bbbk\IO_{n}$ for every
field $\Bbbk$. So we end up with the following result.
\begin{prop}
\label{prop:Iso_of_algebras_fixed_point}Let $\Bbbk$ be a field,
there is an isomorphism of algebras 
\[
\Bbbk\OF_{n+1}\simeq\Bbbk\IO_{n}.
\]
\end{prop}

\begin{rem}
Note that we have an explicit description of this isomorphism. An
isomorphism $\varphi:\Bbbk\OF_{n+1}\to\Bbbk\mathcal{C}(\OF_{n+1})$
is given in \thmref{iso_theorem} and according to \cite{Steinberg2006},
an isomorphism $\mbox{\ensuremath{\psi:\mathcal{C}(\Bbbk\IO_{n})\to\Bbbk\IO_{n}}}$
can be defined by 
\[
\psi(C(\theta))=\sum_{\tau\leq\theta}\mu(\tau,\theta)\tau
\]
where $\leq$ is the standard partial order on an inverse semigroup
and $\mu$ is its related Möbius function (see \cite[Chapter 3]{Stanley1997}).
\end{rem}

As a final observation, we note that a similar isomorphism holds for
certain submonoids. Consider the Catalan monoid $\C_{n}\subseteq\OF_{n}$.
The objects of $\mathcal{C}(\C_{n})$ are sets $X\subseteq[n-1]$
and there is a unique morphism from $X$ to $Y$ if ($|X|=|Y|$ and)
$X\leq Y$. Now consider the monoid $\IC_{n}\subseteq\IO_{n}$ of
all order-preserving and order-increasing partial permutations (see
\cite[Chapter 14]{Ganyushkin2009b}). In other words, a partial permutation
$\theta$ belongs to $\IC_{n}$ if $\theta\in\IO_{n}$ and $i\leq\theta(i)$
for every $i$ in the domain of $\theta$. The semigroup $\IC_{n}$
is not inverse. However, since $E(\IC_{n})=E(\IO_{n})$ it is clear
that it is a reduced $E$-Fountain semigroup which satisfies both
the congruence and the right\textbackslash left ample conditions.
It is also clear that $\theta_{X,Y}\in\IO_{n}$ belongs to $\IC_{n}$
if and only if $X\leq Y$. Therefore, $\mathcal{C}(\IC_{n})$ is isomorphic
to $\mathcal{C}(\C_{n+1})$ and we obtain:
\begin{cor}
\footnote{This is an unpublished observation of the authors of \cite{Margolis2018A}
(private communication).}Let $\Bbbk$ be a field, there is an isomorphism of algebras 
\[
\Bbbk\C_{n+1}\simeq\Bbbk\IC_{n}.
\]
\end{cor}

\bibliographystyle{plain}
\bibliography{library}

\end{document}